\documentclass[12pt]{article}

\usepackage[latin1]{inputenc}
\usepackage[cmex10]{amsmath}
\usepackage{amsthm}
\usepackage{hyperref}
\usepackage{amsfonts}
\usepackage{xcolor}
\usepackage{amssymb}
\usepackage{graphics}
\usepackage[top=1.03in, bottom=1.03in, left=1.03in, right=1.03in]{geometry}
\usepackage{amssymb,amsmath,amsthm,amscd,epsf,latexsym,verbatim,graphicx,amsfonts}
\usepackage{authblk}
\usepackage{tikz}

\def\Znum{{\mathbb{Z}}} % integers
\def\Qnum{{\mathbb{Q}}} % reals

\newtheorem{theorem}{Theorem}[section]
\newtheorem{lemma}[theorem]{Lemma}
\newtheorem{corollary}[theorem]{Corollary}
\newtheorem{proposition}[theorem]{Proposition}

\begin{document}

\title{Signed Difference Sets}

\author{
Daniel M.\ Gordon\\
IDA Center for Communications Research\\
4320 Westerra Court\\
San Diego, CA 92121\\
\tt{gordon@ccr-lajolla.org}
}

\date{\today}

\maketitle

\abstract{
A $(v,k,\lambda)$ difference set in a group $G$ of order $v$ is a subset
$\{d_1, d_2, \ldots,d_k\}$ of $G$ such that $D=\sum d_i$ in the group
ring $\Znum[G]$ satisfies 
$$D D^{-1} = n + \lambda G,$$
where
$n=k-\lambda$.
If $D=\sum s_i d_i$, where the $s_i \in \{ \pm 1\}$, satisfies the
same equation, we will call it a {\em signed difference set}.  

This generalizes both difference sets (all $s_i=1$) and circulant
weighing matrices ($G$ cyclic and $\lambda=0$).  We will show that
there are other cases of interest, and give some results on their existence.
}

\section{Introduction}

Let $G$ be a finite group $G$ of order $v$, and
$\Znum[G]=\{\sum_{g\in G}a_gg\mid a_g\in \Znum\}$ denote the group ring of $G$
over $\Znum$.  
We will identify a set $D=\{d_1, d_2, \ldots,d_k\} \subset G$ with its
group ring element $\sum_{d_i\in D} d_i$.
A $k$-subset $D$ of $G$ is a 
$(v,k,\lambda)$ difference set if
\begin{equation}\label{eq:ds}
D D^{-1} = n + \lambda G,  
\end{equation}
where $n=k-\lambda$.
There is a large literature on difference sets.  See, for example,
\cite{bjl} or  \cite{crcds}.

A circulant weighing matrix $W(n,k)$ is a square $n \times n$ $\{-1,0,1\}$ matrix, where each row
is a right cyclic shift of the one before, and the matrix satisfies the
equation
\begin{equation}\label{eq:cwm}
W W^{T} = kI_n,  
\end{equation}
where $W^T$ is the transpose of $T$.  
See \cite{tan} for background on circulant weighing matrices.
The set of locations of the $1$'s in the first row is denoted by
$P$, and the locations of the $-1$'s by $N$.  
It is well known that they only exist when $k=s^2$ is a perfect
square, and without loss of generality we have $\lvert P \rvert =(k+s)/2$, and $\lvert N \rvert =(k-s)/2$.
Writing (\ref{eq:cwm}) as an equation in the group ring of 
$G = \Znum_n$, we have
\begin{equation}\label{eq:pn}
(P-N)(P-N)^{-1} = k.  
\end{equation}

The website \cite{ljcr} contains an online database with existence
results for difference sets and circulant weighing matrices over a
wide range of parameters.

Group developed weighing matrices, the noncyclic analogue of circulant weighing matrices, have been
considered, see  \cite{ad99} and \cite{arasu2013group}.

If 
$$D=\sum s_i d_i = P-N,$$
where the $s_i \in \{ \pm 1\}$, satisfies (\ref{eq:ds}),
we will call it a {\em signed difference set}.  
As for circulant weighing matrices, we will adopt the convention that $\lvert P \rvert  \geq \lvert N \rvert $.
This generalizes both difference sets ($N=\emptyset$) and group developed
weighing matrices ($\lambda=0$).  We will show that
there are other cases of interest, and give some results on their
existence.

For group elements $g$ not in $P$ or $N$ we will set $s_g=0$, so that 
$D = \sum_{g \in G} s_g g$.

For example, Theorem~\ref{thm:v-1} in the next section shows that
$P=\{1,2,4\}$ and $N=\{3,5,6\}$ form a $(7,6,-1)$ signed difference
set in $\Znum_7$.

\begin{lemma}\label{lem:p}
  Suppose $D = P-N$ is a signed difference set
Then $s = \sqrt{\lambda (v-1)+k}$ must be an integer, and
$$
\lvert P \rvert  = \frac{k + s}{2}.
$$
and
$$
\lvert N \rvert  = \frac{k - s}{2}.
$$

\end{lemma}

\begin{proof}
The lemma follows from counting the number of nonzero differences,
which must sum to $\lambda (v-1)$, and then completing the square.
%Note that for circulant weighing matrices this is the well-known fact
%that their order $k$ must be a square.
\end{proof}

Note that the definition of $s$, $P$ and $N$ 
match up with the terminology for circulant weighing matrices 
and difference sets.
For the former, this is the well-known fact that $k=s^2$ must be a perfect square.
For the latter, $s=k$, and this gives the counting condition 
$\lambda (v-1) = k (k-1)$.

Signed difference sets also can be thought of as periodic sequences, with
applications to signal design.
The {\em periodic autocorrelation} of a
sequence $s$ with period $v$ is the collection of inner products
$$
\theta (\tau) = \sum_{i=0}^{v-1} s_{t+\tau} s^*_{t}
$$
where $*$ denotes complex conjugation, and the sum $t+\tau$ is mod $v$.
A $\{-1,1\}$ periodic sequence with 
$\theta(\tau) = C$ for $1\leq \tau < v$, is said to have {\em
  two-level autocorrelation}.  Such sequences correspond to cyclic
difference sets.
If $C=-1$ for $v$ odd or $0$ for $v$ even, the sequence is said to
have {\em ideal} two-level autocorrelation.
See \cite{golomb2005signal} and
\cite{helleseth1998sequences} for details.  In \cite{HuGong2010}
ternary sequences with entries in $\{-1,0,2\}$ with two-level
autocorrelation are constructed.

\section{Simple Examples}\label{sec:ex}

We begin with some trivial constructions.  We will exclude $(v,v,v)$
(every element of $G$ with sign 1) and $(v,v,v-4)$ (every element of
$G$, one with sign $-1$ and the rest $1$).

\begin{proposition}\label{thm:v}
If a $(v,k,\lambda)$ difference set exists in $G$, then a
$(v,v,v-4n)$ signed difference set exists.
\end{proposition}

\begin{proof}
Let $P$ be the $(v,k,\lambda)$ difference set $D$, and $N$ be
the complement of $D$, which is a 
$(v,v-k,v-2k+\lambda)$ difference set.
We have
$$
(P-N) (P-N)^{-1}  = \sum_{i,j\in P} (i-j) +\sum_{i,j\in N} (i-j) 
- \left(\sum_{i\in P,j\in N} (i-j) + \sum_{i\in N,j\in P} (i-j)\right) 
$$
The coefficient of any nonzero element is $\lambda$ from the first
term, $v-2k+\lambda$ from the second, and 
$v-(\lambda + (v-2k+\lambda))$ from the last two (since there are $v$
ways to express any nonzero element as $i-j$, and for each either $i$
and $j$ are in the same set or not).  This adds to
$$
\lambda + (v-2k+\lambda) - \left( v-\lambda-(v-2k+\lambda) \right)  =
v-4(k-\lambda) = v-4n.
$$
\end{proof}

These are also binary periodic sequences with two-level
autocorrelation.  The sequences corresponding to $(4n-1,2n-1,n-1)$ difference sets
have ideal autocorrelation.

\begin{proposition}\label{thm:v-1}
  For $v$ an odd prime, taking 
$P$ to be the squares modulo $v$ 
and
$N$ the nonsquares
gives a $(v,v-1,-1)$ signed difference set.
\end{proposition}

\begin{proof}
Write $G = \Znum_v$ as $\left\{ e^{2\pi i r/v}, r =
0,1,\ldots v-1 \right\}$.  Then 
$$
(P-N) = \sum_{r=0}^{v-1} \left(\frac{r}{p} \right) e^{2\pi i r/v}.
$$
The right side is just the Gauss sum $G_v(\chi)$ for the quadratic character
$\chi = \left(\frac{r}{v} \right)$.  Gauss famously showed that 
$$
G_v(\chi) = \left\{
\begin{array}{cc}
\sqrt{v} & v \equiv 1 \pmod 4\\
i\sqrt{v} & v \equiv 3 \pmod 4
\end{array}
\right.  
$$
See, for example, Theorem 1.2.4 of \cite{gauss}.  Therefore
$$
G_v(\chi) \overline{G_v(\chi)} = v,
$$
and since $v$ is prime the elements $e^{2 \pi i r/v}$ are linearly
independent over $\Qnum$, so
$$
(P-N)(P-N)^{-1} = v + \lambda G
$$
for some $\lambda$.  From Lemma~\ref{lem:p}, we have
$$s = \sqrt{\lambda(v-1)+k}= \sqrt{(\lambda+1)(v-1)} = 0,$$
so $\lambda = -1$.
\end{proof}

There are no signed difference sets with $\lambda < -1$;
using Lemma~\ref{lem:p},
\begin{eqnarray*}
\sum_{\begin{array}{cc}0 \leq i,j<k\\i \neq j\end{array}}
%\sum_{i \neq j \in D} 
s_is_j ((d_i-d_j) \bmod v) & = & \lvert P \rvert \cdot (\lvert P \rvert -1) +
\lvert N \rvert \cdot (\lvert N \rvert -1) - 2 \lvert P \rvert  \cdot \lvert N \rvert  \\
&=& s^2-k = \lambda (v-1).
\end{eqnarray*}
The only solutions to this with $\lambda<0$ are $s=1,k=v$, as in
Proposition~\ref{thm:v} when $D$ is a Paley difference set,
and $s=0, k = v-1$, as in Proposition~\ref{thm:v-1}.

\begin{proposition}\label{thm:paley}
For $v=4n-1$ a prime power there exists a $(4n-1,2n,n-2)$ signed
difference set.
\end{proposition}

\begin{proof}
For such $v$ a $(4n-1,2n-1,n-1)$ Paley
difference set $D$ exists, where the elements of the difference set are
the squares modulo $v$.  Since $-1$ is a nonsquare modulo $v$, for every $i$ 
exactly one of $i$ or $-i$ is a square modulo $v$, so 
$P=D$ and $N=\{0\}$ gives a
$(4n-1,2n,n-2)$ signed difference set.
\end{proof}

\section{Residue Signed Difference Sets}\label{sec:res}

Propositions~\ref{thm:v-1} and \ref{thm:paley} give signed difference sets
arising from quadratic residues.  Difference sets arise from $e$th
power residues for $e=2,4$ and $8$ \cite{lehmer}.  Here we give a
signed difference set formed using fourth power residues.

Let $H_4$ denote the fourth power residues modulo $v$.
Theorem 5.3.4(b) of \cite{gauss} shows that for $v$ a prime with 
$v=9+4b^2$ for an odd integer $b$, the fourth powers mod $v$ together
with $0$ form a difference set.
This proof will be directly analogous to the proof of that theorem.

Define
$$S(4) = \sum_{r \in H_4} e^{2\pi i r/v}.$$
From Lemma 5.1.1 of \cite{gauss} we have
\begin{equation}\label{eq:s4}
  S(4) = \frac{g(4)-1}{4},
\end{equation}
where the $g(k)$ is the Gauss sum
$$g(k) = \sum_{m=0}^{v-1} e^{2 \pi i m^k/v}.$$

\begin{corollary}\label{cor:g}
Let $v$ be a prime $\equiv 13 \pmod {16}$ of the form 
$25+4y^2$, for an odd integer $y$.  Take
$P$ to be the fourth powers modulo $v$,
$N$ to be $\{0\}$, 
and $D=P-N$.
Then $D$ is a signed $(v,(v+3)/4,(v-13)/16)$ difference set if and only if
  \begin{equation}\label{eqn:g4}
  \lvert g(4)-5 \rvert ^2 = 3v+25.    
  \end{equation}
\end{corollary}

\begin{proof}
Take $G$ to be $\{e^{2 \pi i r/v}\}$ as in Proposition~\ref{thm:v-1}.
and assume that $D$ is a signed difference set.  Then
  \begin{eqnarray*}
D D^{-1}  & = & (S(4)-1)\overline{(S(4)-1)}   \\
& = & \sum_{r,s\in H_4} e^{2\pi i (r-s)/v} - \sum_{r \in H_4} \left(
e^{2\pi i r/v} + e^{-2\pi i r/v}\right) + 1 \\
& = & \frac{v-1}{4} +\lambda \sum_{m \not\equiv 0 \bmod v} e^{2 \pi i
  m/p}  \\
& = & \frac{v-1}{4} -\lambda + 1 = \frac{3v+25}{16}
  \end{eqnarray*}

The corollary then  follows from (\ref{eq:s4}).

Conversely, suppose that (\ref{eqn:g4}) holds.  From the same lemma we
have
$\lvert S(4)-1 \rvert ^2 = (3v+25)/16$.  Rearranging the terms in 
$(S(4)-1)\overline{(S(4)-1)}$, we get
$$
\sum_{m=1}^{v-1} a_m e^{2\pi i m/v} - \sum_{r \in H_4} \left(e^{2\pi i   r/v}+ e^{-2\pi i r/v}\right) = \lambda \sum_{m=1}^{v-1} e^{2\pi  i m/v},
$$
where $a_m$ is the number of ways of representing  $m$ as a difference
of two elements of $H_4$.  
Since the elements $e^{2 \pi i m/v}$ are linearly independent over
$\Qnum$, each term must be equal to $\lambda$, so $D$ is a difference
set.
\end{proof}

\begin{theorem}
  For $v$ and $D$ as in Corollary~\ref{cor:g}, $D$ is a signed
  difference set.
\end{theorem}

\begin{proof}
By Theorem 4.2.1 of \cite{gauss},
$$
g(4) = \sqrt{v} \pm \left( -2(v+25\sqrt{v}) \right).
$$
From this we find 
\begin{eqnarray*}
\lvert g(4)-5 \rvert ^2 & = & g(4) \overline{g(4)} - 5\left( g(4) + \overline{g(4)}\right) + 25 \\  
& = & v + (2v+10\sqrt{v}) - 10 \sqrt{v} + 25 \\
& = & 3v+25,
\end{eqnarray*}
and by the Corollary, $D$ is a signed difference set.
\end{proof}

The first such signed difference sets are $(29,8,1)$, $(61,16,3)$ and $(349,88,21)$.
A search was done for other signed difference sets of $e$th-power residues with $N=\{0\}$
that form an SDS.  There were none with $v<10000$ for $m=3,5,6,8,12$.

\section{Prime Pair Signed Difference Sets}\label{sec:pp}

It is well known that if $q$ and $q+2$ are odd prime powers, then letting $n=(q+1)^2/4$, a 
$(4n-1,2n-1,n-1)$ difference set exists (see \cite{bjl}).  A similar
family of signed difference set exists:

\begin{theorem}\label{thm:pp}
Let $n = m^2$, where $q=2m-3$ and $r=q+6=2m+3$ are prime powers.  Let $\chi$
denote the quadratic character in any finite field: $\chi(x)=1$ if $x$
is a square, $\chi(x)= -1$ if $x$ is not a square, and $\chi(0)=0$.
Then 
$N=\{(0,0)\}$ and
$$
P = \{ (x,y): \chi(x) \chi(y) = 1 \} \cup \{ (0,y): y \neq 0 \}
$$
gives a $(4n-9,2n-1,n-1)$ signed difference set in
$G = (GF(q),+) \oplus (GF(r),+)$.
\end{theorem}

\begin{proof}
We have $v = q(q+6) = 4m^2-9= 4n-9$ and it is easy to check that $k = 2m^2-1=2n-1$,

The elements $(x,y)\in P$ with $y \neq 0$ form a multiplicative group
$M$ in the ring $GF(q) \oplus GF(q+6)$, and $P = MP$, so 
differences $(x,y)$ and $(mx,m^* y)$ with $(m,m^*)\in M$ occur equally
often in $\Delta P$.  So there are constants $\lambda_1$, $\lambda_2$,
$\lambda_3$, and $\lambda_4$ such that the number of times each nonzero difference 
$((x,y)$ occurs is:
$$
\begin{array}{rl}
  \lambda_1: &   \chi(x)  =  \chi(y) \neq 0, \\
  \lambda_2: &   \chi(x)  =  -\chi(y) \neq 0, \\
  \lambda_3: &   x  =  0, \\
  \lambda_4: &   y  =  0. \\
\end{array}
$$

Since one of $q$ and $q+6$ is $1 \pmod 4$ and the other is $3 \pmod
4$, $(-1,-1) \not \in M$, so $(x,y)\in P$ if and only if $(-x,-y) \not
\in P$, so $\lambda_1 = \lambda_2$..

To compute $\lambda_3$, every $(0,d)$ occurs as a difference
$(0,y+d)-(0,y)$ or $(x,y)-(x,y^*)$, where $x \neq 0$
and $\chi(y)=\chi(y^*)$.  The number of ways this can happen is
$$
r-2 +  \frac{(r-3)(q-1)}{4} = n+1.
$$
$(0,d)$ also occurs twice as a difference of elements of $P$ and $N$:
$(0,d)-(0,0)$ and $(0,0)-(0,-d)$, so $\lambda_3 = n-1$.

To compute $\lambda_4$, note that every $(d,0)$ occurs as a difference:
$(d,y)-(0,y)$, $(0,y)-(-d,y)$, or $(x,y)-(x,y^*)$, where $x \neq 0$
and $\chi(y)=\chi(y^*)$.  The number of ways this can happen is
$$
2 \  \frac{r-1}{2} + \frac{(r-1)(q-3)}{4} = n-1.
$$
Since no $(d,0) \in P$, there is no contribution from a difference
with $(0,0)$, so $\lambda_4 = n-1$.

Altogether there are $(2n-2)(2n-3)$ differences of elements of $P$.
Subtracting the $2(2n-1)$ differences of an element of $P$ and
$(0,0)$, and the $(n-1)(4n-2)$ from the $\lambda_3$ and $\lambda_4$ terms,
we have a total of $4 (m+1)^2 (m-1)(m-2)$ among the $4(m+1)(m-2)$
$\lambda_1$ and $\lambda_2$ terms, giving $\lambda_1 = \lambda_2 = (m+1)(m-1) = n-1$.
\end{proof}

The first signed difference sets of this type are $(27,17,8)$,
$(55,31,15)$, and $(91,49,24)$.
%\todo{Make $(391,199,99)$ and $(475,241,120)$}

\section{Computer Searches}

Computer searches were done to find signed difference sets.  Just
exhausting over potential sets quickly become too expensive, so the
following theorems were used to speed the search.  As with difference
sets and circulant weighing matrices, multiplier theorems can greatly
cut down the search space.

The following theorem is in McFarland's thesis (see Result 3.2 in \cite{tan}):

\begin{theorem}\label{thm:mult}
  Let $G$ be a finite abelian group of order $v$ and exponent $e$. 
For $D \in \Znum[G]$ such that
$$
D D^{(-1)} \equiv n \bmod G
$$
where $\gcd(v,n)=1$, let
$$
n = p_1^{e_1} \cdots p_s^{e_s} 
$$
where the $p_i$'s are distinct primes.  Let $t$ be an integer with
$\gcd(v,t)=1$, and suppose there are integers $f_1,\ldots,f_s$ such
that
$$
t \equiv p_1^{f_1}  \equiv  \cdots  \equiv  p_s^{f_s}  \bmod e.
$$
Then $t$ is a multiplier of $D$, and some translate of $D$ is fixed by $t$.
\end{theorem}

Another useful tool was used extensively in \cite{agz}.
As with difference sets and circulant weighing matrices, we may use
subgroups of $G$ to restrict possible signed difference sets.
Let $v=dw$, with $d,w>1$, $H$ be a subgroup of $G$ of order $w$, and
$\overline{D}$ be the natural map from $\Znum[G]$ to $\Znum[H]$.

\begin{lemma}\label{lem:subgp}
For a $(v,k,\lambda)$ signed difference set $D$ as above,
\begin{align}
& \overline{D}  =  \sum_{h \in H} b_h h,\\
& \sum_{h \in H} b_h  = \lvert P \rvert -\lvert N \rvert  =  s,\label{eq:s}\\ 
& \sum_{h \in H} b_h^2   =   k+ \lambda(d-1).\label{eq:k}
\end{align}
\end{lemma}

The $b_i$'s are called the {\em intersection numbers}; see Lemma
VI.5.4 of \cite{bjl} and Lemma 2.5 of \cite{agz}.  As in Algorithm 1
of \cite{agz}, we form an array where the columns are orbits of the
multiplier group in $H$, the rows are orbits of the multiplier group
in $G/H$, and the contents of each array element are the orbits in $G$
congruent to the corresponding orbits in the subgroups.  When
Theorem~\ref{thm:mult} implies a nontrivial multiplier group, this
dramatically cuts down the search space.

\begin{table}
\begin{center}
\begin{tabular}{|c|c|c|c||c|c|} \hline
$v$ & $k$ & $\lambda$ & $n$ & $\lvert P \rvert $ & $\lvert N \rvert $ \\ \hline
19 & 13 & 2 &  11 &  10 & 3\\
19 & 13 & 6  & 7 &  12 & 1 \\
20 & 11 & 2  & 9 &  9 & 2 \\
31 & 24 & 4  & 20 &  18 & 6 \\
35 & 19 & 3  & 16 & 15 & 4 \\
35 & 21 & 10 & 11 & 20 & 1 \\
51 & 19 & 3 & 16 & 16 & 3 \\
53 & 40 & 27 & 13 & 39 & 1 \\
55 & 10 & 1 & 9 & 9 & 1 \\
%55 & 31 & 15 & 16 & 30 & 1 \\
67 & 49 & 12 & 37 & 39 & 10 \\
67 & 49 & 20 & 29 & 43 & 6 \\
71 & 51 & 1 & 50 & 21 & 20 \\
73 & 36 & 4 & 32 & 27 & 9 \\
78 & 53 & 28  & 25 & 50 & 3 \\
89 & 33 & 1 & 32 & 22 & 11 \\
%91 & 49 & 24 & 25 & 48 & 1 \\
91 & 76 & 60 & 16 & 75 & 1 \\
93 & 32 & 7 & 25 & 29 & 3 \\
93 & 73 & 48 & 25 & 70 & 3 \\
104 & 29 & 4 & 25 & 25 & 4 \\
111 & 66 & 17 & 49 & 55 & 11 \\
%187 & 97 & 48 & 49 & 96 & 1 \\
219 & 83 & 19 & 64 & 74 & 9 \\
219 & 172 & 108 & 64 & 163 & 9 \\
247 & 127 & 63  & 64 & 126 & 1 \\

\hline
\end{tabular}

\caption{Sporadic Cyclic Signed Difference Sets}
\label{tab:sporadic}
\end{center}
\end{table}

Table~\ref{tab:sporadic} lists cyclic signed difference sets not given by the
constructions in Sections~\ref{sec:ex}, \ref{sec:res} and \ref{sec:pp} that were found this way.  The
sets themselves are available at \cite{ljcr}.

%\newpage

\section{Noncyclic Signed Difference Sets}

There has been a large amount of work on noncyclic difference sets,
and a smaller number of papers on group developed weighing matrices,
the generalization of circulant weighing matrices to an abelian group.

Arasu and Hollon \cite{arasu2013group} gather known results about such
weighing matrices, and settle the existence question for many
parameters.
They give tables for weighing matrices $C(v,k)$ in cyclic groups with $v,k \leq 100$.
The existence question is settled for all but 98 out of 1022 such parameters.

A search for noncyclic signed difference sets with $\lambda \neq 0$
not constructed with Theorem~\ref{thm:pp} turned
up only one such signed difference set:
an $(18,13,4)$ set in $\Znum_2 \times \Znum_3 \times \Znum_3$.

%P=[[0,0,2],[0,0,1],[0,1,2],[0,1,1],[0,2,2],[0,2,1],[1,0,1],[1,1,0],[1,2,0],[1,2,1],[1,2,2]]
%N=[[1,0,0],[1,1,1]]
%Z=[[0,0,2],[0,1,2],[0,2,2],[1,0,1],[1,1,1]]

\begin{theorem}
In the group $G=\Znum_2 \times \Znum_3 \times \Znum_3$, let $P$ be
$$\left\{(0,x,1), (0,x,2), (1,2,x) :  x \in \{0,1,2\} \right\} \cup \{(1,0,1), (1,1,0)\},$$
%$(0,x,1)$, $(0,x,2)$, $(1,2,x)$ for $x \in \{0,1,2\}$, $(1,0,1)$ and $(1,1,0)$, 
and let
$N=\left\{ (1,0,0), (1,1,1) \right\}$.
Then $D=P-N$ is a signed difference set.
\end{theorem}

\begin{figure}
\begin{center}
\begin{tikzpicture}
\draw [very thick] (0,0) rectangle (1,1);
\draw [very thick] (0,1) rectangle (1,2);
\draw [very thick] (0,2) rectangle (1,3);
\draw [very thick,fill=red] (1,0) rectangle (2,1);
\draw [very thick,fill=red] (1,1) rectangle (2,2);
\draw [very thick,fill=red] (1,2) rectangle (2,3);
\draw [very thick,fill=red] (2,0) rectangle (3,1);
\draw [very thick,fill=red] (2,1) rectangle (3,2);
\draw [very thick,fill=red] (2,2) rectangle (3,3);

\draw [very thick,fill=red] (6,0) rectangle (7,1);
\draw [very thick,fill=red] (6,1) rectangle (7,2);
\draw [very thick,fill=blue] (6,2) rectangle (7,3);
\draw [very thick,fill=red] (7,0) rectangle (8,1);
\draw [very thick,fill=blue] (7,1) rectangle (8,2);
\draw [very thick,fill=red] (7,2) rectangle (8,3);
\draw [very thick,fill=red] (8,0) rectangle (9,1);
\draw [very thick] (8,1) rectangle (9,2);
\draw [very thick] (8,2) rectangle (9,3);

\draw (0.5,3.5) node {0};
\draw (1.5,3.5) node {1};
\draw (2.5,3.5) node {2};
\draw (6.5,3.5) node {0};
\draw (7.5,3.5) node {1};
\draw (8.5,3.5) node {2};

\draw (-0.5,0.5) node {2};
\draw (-0.5,1.5) node {1};
\draw (-0.5,2.5) node {0};
\draw (5.5,0.5) node {2};
\draw (5.5,1.5) node {1};
\draw (5.5,2.5) node {0};

\draw (1.5,4.5) node {$(0,y,z)$};
\draw (7.5,4.5) node {$(1,y,z)$};
\end{tikzpicture}
\end{center}
  \caption{An $(18,13,4)$ signed difference set in $\Znum_2 \times \Znum_3 \times \Znum_3$,
Elements in $P$ are red, elements of $N$ are blue.}
  \label{fig:233}
\end{figure}
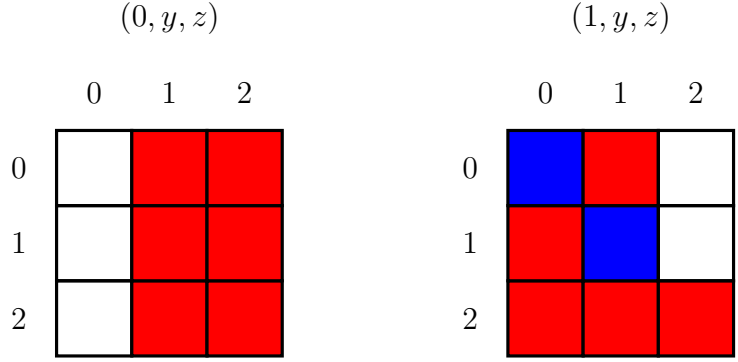

\begin{proof}
Figure~\ref{fig:233} shows this signed difference set, with the two
slices $(0,y,z)$ and $(1,y,z)$ drawn next to each other.  To
demonstrate that it is a signed difference set, we need to show that
for any $(a,b,c)\neq (0,0,0) \in G$, we have

\begin{equation}\label{eq:233}
\sum_{{(x,y,z)} \in G} s_{(x,y,z)} s_{{(x,y,z)} + (a,b,c)} = 4.  
\end{equation}

First consider the case $a=1$.  Since $s_{(0,y,z)}=1$ for $y=1,2$ and
zero otherwise, (\ref{eq:233}) becomes
$$
 \sum_{y=1}^2  \sum_{z=0}^2 s_{(1,y+b,z+c)} + s_{(1,y-b,z-c)}.
$$
All three columns of the right side of Figure~\ref{fig:233} sum
to one, the inner sums over $z$ are each one, and (\ref{eq:233}) follows.

Now consider $a=0$.  The contribution from $x=0$ will be 6 if $b=0$,
and 3 otherwise (in the former case the sum over $z$ is the dot
product of each of the
$y=1$ and $y=2$ columns with itself, and in the latter the dot product
of one with the other).

If $b=0$, the contribution from $x=1$ will be $-2$.  Since $(a,b,c)
\neq (0,0,0)$,
$c \neq 0$, so we are
taking the dot product of each column with a shift of itself, giving
$-1, -1$ and $0$.

If $b \neq 0$ we need to show that the contribution from $x=1$ is $1$.
There are three cases:

\begin{description}
  \item [$c=0$] dot products of adjacent columns, giving $-1, 1$ and $1$.  
  \item [$c=b$] dot products of the shifted adjacent columns are $3, -1$ and $-1$.
  \item [$c=-b$] dot products of adjacent columns shifted the other
    way are $-1,1$ and $1$.
\end{description}
\end{proof}

\section*{Data Availability}

The datasets generated during this research are available at the
author's website: \url{https://dmgordon.org}.

\bibliography{sds}
\bibliographystyle{plain}
\end{document}